\documentclass[11pt]{amsart}
\usepackage{amsmath,paralist,amssymb}
\usepackage{amsthm}
\usepackage{hyperref}
\usepackage{graphicx,subfigure}
\usepackage{tikz,color}
\usepackage[margin=1in]{geometry}
\usetikzlibrary{arrows,shapes,snakes,backgrounds} 
\usepackage{verbatim}

\usepackage[all]{xy}

\definecolor{DarkBlue}{rgb}{0,0,0.8} 
\definecolor{DarkGreen}{rgb}{0,0.5,0.0}

\newtheorem*{lemma*}{Lemma}

\newtheorem*{corollary*}{Corollary}
\newtheorem*{theorem*}{Theorem}
\newtheorem*{theorem1*}{Theorem \ref{thm:vol}}
\newtheorem*{theorem2*}{Theorem \ref{thm:vol2}}
\newtheorem*{theorem3*}{Theorem \ref{thm:volC}}

\numberwithin{equation}{section}

\newtheorem{theorem}{Theorem}[section]
\newtheorem{proposition}[theorem]{Proposition}
\newtheorem{lemma}[theorem]{Lemma}
\newtheorem{corollary}[theorem]{Corollary}
\newtheorem{conjecture}[theorem]{Conjecture} 
\theoremstyle{remark}

\newtheorem{example}{Example}[section]
\theoremstyle{definition}
\newtheorem{definition}[equation]{Definition}

\newcommand\Cat{\operatorname{Cat}}
\newcommand\CT{\operatorname{CT}}
\newcommand\vol{\operatorname{vol}}

\def\ee{{\bf e}}
  \def\e{\epsilon}

\def\indeg{{\rm indeg}} 
 \def\vol{{\rm vol}}
 
\def\vv{{\rm v}}

\def\v{{\bf a}}

 \def\f_H{{\bf w}}
 \def\f{{\bf f}}
 
\def\R{\mathbb{R}}

\def\Z{\mathbb{Z}}
\def\g{{\bf a}}
\def\inc{\textrm{inc}}
 \def\F{\mathcal{F}}
\def\O{\mathcal{O}}
\def\I{\mathcal{I}}

\def\ee{{\bf e}}
  \def\e{\epsilon}
   \def\vol{{\rm vol}}
 
\def\vv{{\rm v}}

\def\aa{{\bf a}}
 \def\f_H{{\bf w}}
 \def\f{{\bf f}}
 
\def\R{\mathbb{R}}

\def\Z{\mathbb{Z}}
\def\g{{\bf a}}
\def\inc{\textrm{inc}}
 \def\F{\mathcal{F}}
\def\O{\mathcal{O}}
\def\I{\mathcal{I}}

\def\CT{CT}
\def\ee{{\bf e}}
  \def\e{\epsilon}
   \def\vol{{\rm vol}}
 
\def\vv{{\rm v}}

\def\v{{\bf a}}
\def\aa{{\bf a}}
 \def\f_H{{\bf w}}
 \def\f{{\bf f}}
 
\def\vvv{{\bf v}}
\def\R{\mathbb{R}}

\def\Z{\mathbb{Z}}
\def\g{{\bf a}}
\def\inc{\textrm{inc}}
 \def\F{\mathcal{F}}
\def\O{\mathcal{O}}
\def\I{\mathcal{I}}

\newcommand{\be}{\begin{equation}}

\newcommand{\eee}{\end{equation}}

\newcommand{\bd}{\begin{definition}}
\newcommand{\ed}{\end{definition}}
\newcommand{\bt}{\begin{theorem}}
\newcommand{\et}{\end{theorem}}
\newcommand{\bl}{\begin{lemma}}
\newcommand{\el}{\end{lemma}}
\newcommand{\bp}{\begin{proposition}}
\newcommand{\ep}{\end{proposition}}
\newcommand{\bc}{\begin{corollary}}
\newcommand{\ec}{\end{corollary}}

\def\R{\mathbb{R}}
\def\v{{\rm v}}
\begin{document}

\tikzstyle{w}=[label=right:$\textcolor{red}{\cdots}$] 
\tikzstyle{b}=[label=right:$\cdot\,\textcolor{red}{\cdot}\,\cdot$] 
\tikzstyle{bb}=[circle,draw=black!90,fill=black!100,thick,inner sep=1pt,minimum width=3pt] 
\tikzstyle{bb2}=[circle,draw=black!90,fill=black!100,thick,inner sep=1pt,minimum width=2pt] 
\tikzstyle{b2}=[label=right:$\cdots$] 
\tikzstyle{w2}=[]
\tikzstyle{vw}=[label=above:$\textcolor{red}{\vdots}$] 
\tikzstyle{vb}=[label=above:$\vdots$] 

\tikzstyle{level 1}=[level distance=3.5cm, sibling distance=3.5cm]
\tikzstyle{level 2}=[level distance=3.5cm, sibling distance=2cm]

\tikzstyle{bag} = [text width=4em, text centered]
\tikzstyle{end} = [circle, minimum width=3pt,fill, inner sep=0pt]

\title[Volumes of generalized Chan-Robbins-Yuen polytopes]{Volumes of generalized Chan-Robbins-Yuen polytopes}
\author{Sylvie Corteel}
\address{Sylvie Corteel, IRIF, CNRS et Universit\'e Paris Diderot, 75205 Paris Cedex 13, France.
{corteel@irif.fr}}

\author{Jang Soo Kim}
\address{Sungkyunkwan University,
2066 Seobu-ro, Jangan-gu, 
Suwon, Gyeonggi-do 16419,
South Korea. {jangsookim@skku.edu}}

\author{Karola M\'esz\'aros}
\address{Karola M\'esz\'aros, Department of Mathematics, Cornell University, Ithaca NY 14853.  \newline{ karola@math.cornell.edu}
}

\date{\today}

\begin{abstract} The normalized volume of the Chan-Robbins-Yuen   polytope ($CRY_n$) is the product of consecutive Catalan numbers. The  polytope $CRY_n$ has captivated combinatorial audiences for over a decade, as there is no combinatorial proof for its volume formula. In their quest to understand $CRY_n$  better, the third author and Morales introduced two natural generalizations of it and conjectured that their volumes are certain powers of $2$ multiplied by a product of consecutive Catalan numbers.  Zeilberger proved one of these conjectures. In this paper we present proofs of both conjectures. 
\end{abstract}

\maketitle

\section{Introduction}
\label{sec:intro}

The Chan-Robbins-Yuen polytope ($CRY_n$)  has captivated combinatorialists for nearly two decades since its introduction in \cite{cry}. Chan, Robbins and Yuen defined $CRY_n$ as the convex hull of the set of $n\times n$ permutation matrices $\pi$ with $\pi_{ij}=0$ if $j\geq i+2$. The polytope $CRY_n$  is  integrally equivalent to  the (type A) flow polytope of the complete graph $K_{n+1}$ with netflow vector  $(1,0, \ldots, 0, -1)$ \cite{mm}. (We define these in Section \ref{sec:flowsdefs}.) Recall  that integer polytopes $\mathcal{P}\subset \R^m$ and $\mathcal{Q}\subset \R^k$ are \textbf{integrally equivalent} if there is an affine transformation $f:\R^m\rightarrow \R^k$ such that $f$ maps $\mathcal{P}$ bijectively onto $\mathcal{Q}$ and $f$ maps $\Z^m\cap  \operatorname{aff}(\mathcal{P})$ bijectively onto $\Z^k\cap \operatorname{aff}(\mathcal{Q})$, where  $ \operatorname{aff}$ denotes affine span. If two polytopes are  integrally equivalent, then they have the same combinatorial type as well as the same volume and more generally the same Ehrhart polynomial.

  Recall  that the \textbf{Ehrhart polynomial} $i(\mathcal{P}, t)$ of an integer polytope $\mathcal{P} \subset \R^m$ counts the  number of integer points of dilations of  the polytope, $i(\mathcal{P}, t):=\#(t \mathcal{P}\cap \Z^m)$. Its leading coefficient is the {\bf volume} of the polytope.   The  {\bf normalized volume} $\vol(P)$ of a 
$d$-dimensional polytope $\mathcal{P} \subset \mathbb{R}^m$ is the volume form which 
 assigns a volume of one to the smallest $d$-dimensional integer simplex in the affine span of $\mathcal{P}$. In other words, the normalized volume of a
$d$-dimensional polytope $\mathcal{P}$ is $d!$ times its  volume.

The polytope $CRY_n$ is  a face of the Birkhoff polytope, the polytope of all doubly stochastic matrices, prominent in combinatorial optimization. Remarkably, the normalized volume of the $CRY_n$ polytope is the product of the first $n-2$ Catalan numbers, as conjectured by Chan, Robbins and Yuen in \cite{cry} and proved by Zeilberger in \cite{Z}.

 \begin{theorem} \label{thm:cry} \cite{cry, Z} The normalized volume of $CRY_n$ is \begin{equation}  \vol(CRY_n)=\prod_{i=1}^{n-2} \Cat(i),\end{equation} where $\Cat(i)=\frac{1}{i+1}{{2i} \choose {i}}$ is the Catalan number.
 \end{theorem}
  
 Zeilberger proved Theorem \ref{thm:cry} analytically via constant terms identities.  Despite the combinatorial volume formula, his theorem still lacks a combinatorial proof. In a quest to broaden the view on  $CRY_n$ and flow polytopes in general, the third author and Morales  introduced and studied signed  flow polytopes in \cite{mm}, and defined   types $C$ and $D$ analogues of the Chan-Robbins-Yuen polytope, $CRYC_{n+1}$ and $CRYD_{n+1}$. They conjectured:

\begin{conjecture} \cite[Conjecture 7.6]{mm}, \cite[Zeilberger's theorem]{Z1} \label{conjcryD} Let $CRYD_{n+1}$ be the flow polytope $\mathcal{F}_{K_{n+1}^{D}}(2,0,\ldots,0)$ where $K^{D}_{n+1}$ is the complete signed graph with $n+1$ vertices (all edges of the form $(i,j,\pm)$, $1\leq i<j\leq n+1$). Then the normalized volume of $CRYD_{n+1}$ is
\[
\vol(CRYD_{n+1}) = 2^{(n-1)^2} \prod_{k=0}^{n-1} \Cat(k).
\]
\end{conjecture}

\begin{conjecture} \cite[Conjecture 7.8]{mm} \label{conjcryC}   Let $CRYC_{n+1}$ be the flow polytope $\mathcal{F}_{K_{n+1}^{C}}(2,0,\ldots,0)$ where $K^{C}_{n+1}$ is the complete signed graph with $n+1$ vertices (all edges of the form $(i,j,\pm)$ for $1\leq i< j\leq n+1$ and $(i,i,+)$ for $1\le i\le n$). Then the normalized volume of $CRYC_{n+1}$ is
\begin{equation}
\vol(CRYC_{n+1}) = 2^{n(n-1)} \prod_{k=0}^{n-1} \Cat(k).
\end{equation}
\end{conjecture}

For details on notation in the above conjectures consult Section \ref{sec:flowsdefs}.  We note that in \cite[p. 834, Conjecture 7.6]{mm} the formula for $\vol(CRYC_{n+1})$ has a typo giving an additional factor of 2. 

\medskip

In \cite{Z1} Zeilberger  proved Conjecture \ref{conjcryD}. In this paper we prove  Conjecture \ref{conjcryC}, by understanding the volume of $CRYC_{n+1}$ in combinatorial terms and translating this understanding to a new constant term identity which we prove with analytic tools. We also give a detailed proof of  Zeilberger's theorem \cite{Z1}, formerly Conjecture \ref{conjcryD}. 

\medskip

\textbf{The outline of the paper} is as follows. In Section \ref{sec:flowsdefs} we give the background on type C flow polytopes (of which the type D flow polytopes are a special case where the graph has no loops) and define  $CRYC_{n+1}$ and $CRYD_{n+1}$. In Section \ref{sec:volumes} we explain how to express the volumes of $CRYC_{n+1}$ and $CRYD_{n+1}$ as constant term identities. In Section \ref{sec:vol} we prove Conjecture \ref{conjcryC} using our insights from Section \ref{sec:volumes} and constant term identity techniques. We  also present a proof of Conjecture \ref{conjcryD} for completeness. In Section \ref{sec:conc} we conclude by a discussion of open problems.

 \section{Type $C_{n+1}$ flow polytopes}
\label{sec:flowsdefs}

Much of this section  follows  the exposition in \cite{mm}. The figures are also borrowed from \cite{mm} with permission. For further details see \cite{mm}.

\subsection{Signed graphs, Kostant partition functions and flows} 
We consider  \textbf{signed graphs} $G$ on the vertex set $[n+1]:=\{1,2,\dots,n+1\}$, which are graphs such that there is a sign $\e \in \{+, -\}$ assigned to each of their edges. We allow loops and multiple edges. The sign of a loop is always $+$, and a loop at vertex $i$ is denoted by $(i, i, +)$. Denote by $(i, j, -)$ and $(i, j, +)$, $i < j$, a negative and a positive edge between vertices $i$ and $j$, respectively.  A positive edge, that is an edge labeled by $+$, is {\bf positively incident}, or, {\bf incident with a positive sign}, to both of its endpoints.  A negative  edge is positively incident to its smaller endpoint and  {\bf negatively  incident} to  its greater endpoint. 
Denote by $m_{ij}^\e$ the multiplicity of edge $(i, j, \e)$ in $G$, $i\leq j$, $\e \in \{+, -\}$. 
To each edge $(i, j, \e)$, $i\leq j$,  of $G$,  associate the positive
type $C_{n+1}$ root $\vv(i,j, \e)$, where $\vv(i,j, -)=\ee_i-\ee_j$
and $\vv(i,j, +)=\ee_i+\ee_j$. Let $S_G := \{\{\vvv_1, \ldots, \vvv_N\}\}$ be the
multiset of roots corresponding to the multiset of edges of $G$. Note that $N=\sum_{1\leq i\leq j\leq n+1}
(m_{ij}^-+m_{ij}^+)$.

 For a signed graph $G$ the {\bf Kostant partition function}  $K_G$ evaluated at the vector $\v \in \Z^{n+1}$ is defined as

$$K_G(\v)= \# \Big\{ (b_{k})_{k \in [N]} \Bigm\vert \sum_{k \in [N]} b_{k}  \vvv_k =\v \textrm{ and } b_{k} \in \Z_{\geq 0}\Big\}.$$

That is, $K_G(\v)$ is the number of ways to write the vector $\v$ as
an $\mathbb{N}$-linear combination of the positive type $C_{n+1}$
roots $\vvv_k$ corresponding to the edges of $G$, without regard to order.  

In this paper positive edges will be colored red and negative edges will be colored black.

 \begin{example}
For the signed graph $G$ in Figure \ref{AAA}, $K_G(1,3,-2)=3,$ since $(1,3,-2)=(\ee_1-\ee_3) + (2\ee_2) + (\ee_2-\ee_3)=(\ee_1+\ee_2) + 2(\ee_2-\ee_3)=(\ee_1-\ee_2)+(2\ee_2)+2(\ee_2-\ee_3)$.
\end{example}

\begin{figure}
\begin{center}
\subfigure[]{
\includegraphics[width=10cm]{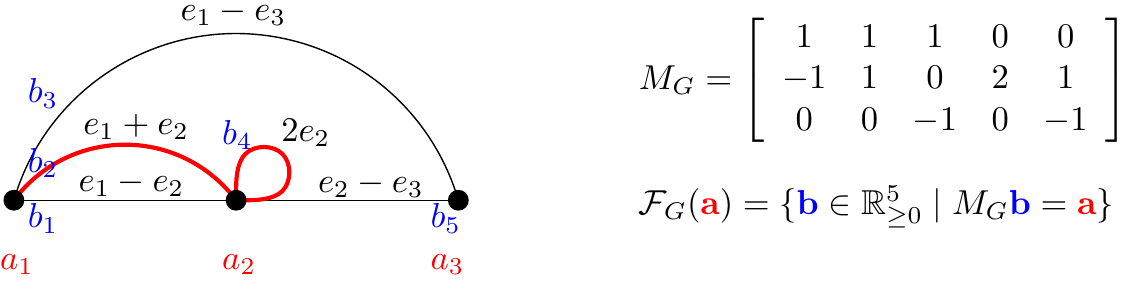}
\label{AAA}
}
\quad
\subfigure[]{
\includegraphics[width=4cm]{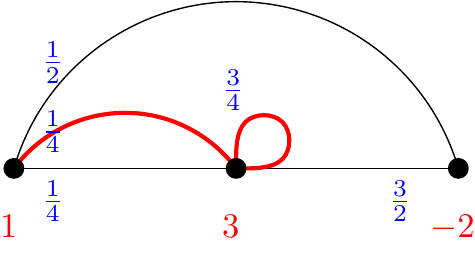}
\label{exflow}
}
\caption{ (a)
A signed graph $G$ on three vertices and the positive roots associated
with each of the five edges. The columns of the matrix $M_G$
correspond to these roots.  The flow polytope $\F_G(\g)$ consists of
the flows ${\bf b}\in \R_{\geq 0}^5$ such that $M_G{\bf b} = {\bf a}$
where ${\bf a}$ is the netflow vector.   The Kostant partition function $K_{G}({\bf a})$ 
counts the lattice points of $\F_G(\g)$, the number of ways of obtaining ${\bf a}$ as an $\mathbb{N}$-integer combination of the roots associated to $G$.\newline
(b) A nonnegative flow on $G$ with  {netflow vector} ${\bf a}=(1,
3,-2)$. The flows on the edges are included. Note that the total flow
on the positive edges is ${\frac14} + {\frac34}=1=\frac12({1}+{3}-{2})$.}
\end{center}
\end{figure}

Let $G$ be a signed graph on the vertex set $[n+1]$, and $M_G$ be the
$(n+1)\times N$ matrix whose columns are the vectors in $S_G$.  Fix an
integer  vector $\g=(a_1, \ldots, a_n, a_{n+1}) \in \Z^{n+1}$ which we
call the {\bf netflow}.  An {\bf $\g$-flow} $\f_G$ on $G$ is a vector $\f_G=(b_k)_{k \in [N]}$, $b_k \in \R_{\geq 0}$  such that $M_G \f_G=\g$. That is, for all $1\leq v \leq n+1$, we have 

 \begin{equation} \label{eqn:flow}
a_v+ \sum_{\substack{e \in E(G),\\ \inc(e, v )=-}} b(e)= \sum_{\substack{e \in E(G),\\ \inc(e, v)=+}} b(e)+\sum_{e=(v, v, + )} b(e),
  \end{equation}
  
\noindent   where $b(e_k)=b_k$, $\inc(e, v)=-$ if $e=(g, v, -)$, $g <v$, and $\inc(e, v)=+$ if $e=(g, v, +)$, $g <v$, or $e=( v, j, \e)$, $v <j,$ and $\e \in \{+, -\}$. 

\begin{example}
Figure \ref{exflow} shows a signed graph $G$ with three vertices with flow assigned to each edge. The netflow is ${\bf a}=(1,3,-2)$. We can check that \eqref{eqn:flow} holds for this example.
Indeed we have $1=\frac{1}{2}+\frac{1}{4}+\frac{1}{4}$, $3+\frac{1}{4}=\frac{1}{4}+\frac{3}{4}+\frac{3}{2}+\frac{3}{4}$, etc.
\end{example}

Call $b(e)$  the {\bf flow} assigned to edge $e$ of $G$. If the edge $e$ is negative, one can think of $b(e)$ units of fluid flowing on $e$ from its smaller to its bigger vertex. If the edge $e$ is positive, then one can think of $b(e)$ units of fluid flowing away both from $e$'s smaller and bigger vertex to ``infinity.'' Edge $e$ is then a ``leak" taking away $2b(e)$ units of fluid.

From the above explanation it is clear that if we are given an $\aa$-flow $\f_G$ such that 
\begin{equation} \label{eq:flowconstraint}
\sum_{i=1}^{n+1}a_i=2y,
\end{equation}
for some positive integer $y$ then  $\sum_{e=(i, j, +)}b(e)=y$. Using again the example on Figure  \ref{exflow}, $ \sum_{i=1}^{n+1}a_i=1+3-2=2\sum_{e=(i, j, +)}b(e)=2\left(\frac{1}{4}+\frac{3}{4}\right)$.
 
 An {\bf  integer} {\bf $\g$-flow} $\f_G$ on $G$ is an $\g$-flow
 $\f_G=(b_i)_{i \in [N]}$, with $b_i \in \Z_{\geq 0}$. It is a
 matter of checking the definitions to see that for a signed graph $G$
 on the vertex set $[n+1]$ and vector  $\g=(a_1, \ldots, a_n, a_{n+1})
 \in \Z^{n+1}$,  the number of  integer $\g$-flows on $G$ is given by
 the Kostant partition function $K_G(\g)$.

Define the {\bf flow polytope} $\F_G(\g)$ associated to a signed graph $G$ on the vertex set $[n+1]$ and the integer vector $\aa=(a_1, \ldots, a_{n+1})$ as the set of all $\g$-flows $\f_G$ on $G$, i.e., $\F_G=\{\f_G \in \R^N_{\geq 0} \mid M_G \f_G = \g\}$. The flow polytope 
  $\F_G(\g)$ then naturally lives in $\mathbb{R}^{N}$, where $N$ is the number  of edges of $G$.

Classical  type $A_n$ flow polytopes are type $C_{n+1}$ flow polytopes $\F_G(\g)$ such that the graph $G$ has only negative edges. 

 From the definition of the Ehrhart polynomial  and the Kostant partition function it follows that
\begin{equation} \label{ehr} \displaystyle i({\F_G(\g)}, t) = K_G(t \g). \end{equation}

\subsection{Chan-Robbins-Yuen polytopes}
 We think of the Chan-Robbins-Yuen polytope $CRY_n$ as the flow polytope of 
 the (unsigned) complete graph on $n+1$ vertices $\mathcal{F}_{K_{n+1}}(1,0,\ldots,0,-1)$ (since they are integrally equivalent).    Zeilberger computed  the normalized volume of this polytope (Theorem \ref{thm:cry}) using the {Morris identity} \cite[Thm. 4.13]{WM}.   

Let $K^{D}_{n+1}$ be the complete signed graph on $n+1$ vertices, that is, its edges are of the form $(i,j,\pm)$ for $1\leq i<j\leq n+1$ corresponding to all the positive roots in type $D_{n+1}$. Let $CRYD_{n+1}=\mathcal{F}_{K_{n+1}^{D}}(2,0,\ldots,0)$ be the type $D$  analogue of the Chan-Robbins-Yuen polytope. Similarly, let $K^{C}_{n+1}$ be the signed graph on $n+1$ vertices with  edges of the form $(i,j,\pm)$ for $1\leq i< j\leq n+1$ and $(i,i,+)$ for $i \in [n+1]$,  corresponding to all the positive roots in type $C_{n+1}$. Let $CRYC_{n+1}=\mathcal{F}_{K_{n+1}^{C}}(2,0,\ldots,0)$  be the type $C$  analogue of the Chan-Robbins-Yuen polytope.  Conjectures  \ref{conjcryD} and  \ref{conjcryC} concern these polytopes, and are the subject of this paper.

\subsection{Dynamic integer flows} 
Given a signed graph $G$ and an edge $e=(i,j,+)$ of $G$, we will regard $e=(i,j,+)$ as two positive {\bf half-edges} $(i,\varnothing,+)$ and $(\varnothing,j,+)$ that still have ``memory'' of being together (see Figure \ref{fig:dynflow} (a)). We assign nonnegative {\bf integer} flows $b_{\ell}(e)$ and $b_{r}(e)$ to the left and right halves of the positive edge, starting at the left half-edge. Once we assign $b_{\ell}(e)$ units of flow, {\bf we add $b_{\ell}(e)$ {\bf extra right positive half-edges}}  {\bf incident to $j$}. Any right positive half-edge $e'$ is assigned a nonnegative integer flow $b_r(e')$ (whether it was an extra right positive half-edge, or an original one). When we assign a nonnegative integer flow to a right positive half-edge no edges of any kind are added making the process of adding extra edges to the graph finite.

An analogue of Equation \eqref{eqn:flow} still holds:
\begin{equation}\label{eqn:dynflow}
a_i+ \sum_{e \in \I_i(G)} b(e)= \sum_{e \in \O_i^-(G)} b(e) + \sum_{e=(i,\cdot,+) \in \O_i^+(G)} b_{\ell}(e)+\sum_{e=(\cdot,i,+)\in \O_i^+(G)} b_{r}(e) + \sum_{\text{extra right} \atop \text{half-edges $e'=(\varnothing,i,+)$}} b_{r}(e'),
\end{equation}
where $a_i$ is the netflow at vertex $i$ and $\I_i(G)$, $\O^-_i(G)$, and $\O^+_i(G)$ are as follows. Given a signed graph $G$ and one of its vertices $i$, let $\mathcal{I}_i=\mathcal{I}_i(G)$ be the multiset of {\bf incoming edges} to $i$, which are defined as negative edges of the form $(\cdot,i,-)$. Let $\mathcal{O}_i=\mathcal{O}_i(G)$ be the multiset of {\bf outgoing edges} from $i$, which are defined as  edges of the form $(\cdot,i, +)$ and $(i,\cdot,\pm)$. Finally,  let $\mathcal{O}_i^{\pm}$ be the signed refinement of $\mathcal{O}_i$. Define $\indeg_G(i):=\#\mathcal{I}_i(G)$ to be  the {\bf indegree} of vertex $i$ in $G$. 

We call the integer ${\bf a}$-flows of equation \eqref{eqn:dynflow} {\bf dynamic}. 

For the signed graph $G$ in Figure \ref{fig:dynflow} (a) with only one positive edge $e=(1,3,+)$, we give three of its $17$ integer dynamic flows with netflow $(2,1,1)$ where we add $b_{\ell}(e)=0,1$ and $2$ right half-edges respectively.

\begin{figure}
\centering
\subfigure[]{
\raisebox{5pt}{
\begin{tikzpicture}[scale=0.45] 
\draw [-] (4,0) to (0,0);
\draw[red] [-] (0,0) arc (180:0:2cm and 1.9cm);
\draw [-] (0,0) arc (180:0:2cm and 1.3cm);
\node [style=bb,label=below:$\textcolor{red}{2}$] (2) at (0, 0) {};
\node [style=bb,label=below:$\textcolor{red}{1}$] (3) at (2, 0) {};
\node [style=bb,label=below:$\textcolor{red}{1}$] (4) at (4, 0) {};
\node [label=below:$G$] (a) at (2, 3.5) {};
\end{tikzpicture}
}
\raisebox{20pt}{$\longrightarrow$}
\raisebox{5pt}{
\begin{tikzpicture}[scale=0.45] 
\draw [-] (4,0) to (0,0);
\draw[red,dotted] [-] (0,0) arc (180:0:2cm and 1.9cm);
\draw[red] [-] (0,0) arc (180:105:2cm and 1.90cm);
\draw[red] [-] (4,0) arc (0:75:2cm and 1.90cm);
\draw [-] (0,0) arc (180:0:2cm and 1.3cm);
\node [style=bb,label=below:$\textcolor{red}{2}$] (2) at (0, 0) {};
\node [style=bb,label=below:$\textcolor{red}{1}$] (3) at (2, 0) {};
\node [style=bb,label=below:$\textcolor{red}{1}$] (4) at (4, 0) {};
\end{tikzpicture}
}
}
\qquad
\subfigure[]{
\raisebox{28pt}{
\begin{tabular}{ccc}
\begin{tikzpicture}[scale=0.45] 
\draw [-] (4,0) to (0,0);
\draw[red,dotted] [-] (0,0) arc (180:0:2cm and 1.9cm);
\draw[red] [-] (0,0) arc (180:105:2cm and 1.90cm);
\draw[red] [-] (4,0) arc (0:75:2cm and 1.90cm);
\draw [-] (0,0) arc (180:0:2cm and 1.3cm);
\node [style=bb,label=below:$\textcolor{red}{2}$] (2) at (0, 0) {};
\node [style=bb,label=below:$\textcolor{red}{1}$] (3) at (2, 0) {};
\node [style=bb,label=below:$\textcolor{red}{1}$] (4) at (4, 0) {};
\node [label=below:$\textcolor{blue}{0}$] (b) at (0, 1.9) {};
\node [label=below:$\textcolor{blue}{1}$] (c) at (2, 1.8) {};
\node [label=below:$\textcolor{blue}{1}$] (c) at (1, 0.4) {};
\node [label=below:$\textcolor{blue}{2}$] (c) at (3, 0.4) {};
\node [label=below:$\textcolor{blue}{4}$] (c) at (2.8, 2.5) {};
\end{tikzpicture}
&
\begin{tikzpicture}[scale=0.45] 
\draw [-] (4,0) to (0,0);
\draw[red,dotted] [-] (0,0) arc (180:0:2cm and 1.9cm);
\draw[red] [-] (0,0) arc (180:105:2cm and 1.90cm);
\draw[red] [-] (4,0) arc (0:75:2cm and 1.90cm);
\draw[red] [-] (4,0) arc (0:75:2cm and 2.5cm);
\draw [-] (0,0) arc (180:0:2cm and 1.3cm);
\node [style=bb,label=below:$\textcolor{red}{2}$] (2) at (0, 0) {};
\node [style=bb,label=below:$\textcolor{red}{1}$] (3) at (2, 0) {};
\node [style=bb,label=below:$\textcolor{red}{1}$] (4) at (4, 0) {};
\node [label=below:$\textcolor{blue}{1}$] (b) at (0, 1.9) {};
\node [label=below:$\textcolor{blue}{1}$] (c) at (2, 1.8) {};
\node [label=below:$\textcolor{blue}{0}$] (c) at (1, 0.4) {};
\node [label=below:$\textcolor{blue}{1}$] (c) at (3, 0.4) {};
\node [label=below:$\textcolor{blue}{2}$] (c) at (2.8, 2.5) {};
\node [label=below:$\textcolor{blue}{1}$] (c) at (2.8, 3.2) {};
\end{tikzpicture}
&
\begin{tikzpicture}[scale=0.45] 
\draw [-] (4,0) to (0,0);
\draw[red,dotted] [-] (0,0) arc (180:0:2cm and 1.9cm);
\draw[red] [-] (0,0) arc (180:105:2cm and 1.90cm);
\draw[red] [-] (4,0) arc (0:75:2cm and 1.90cm);
\draw[red] [-] (4,0) arc (0:75:2cm and 2.5cm);
\draw[red] [-] (4,0) arc (0:75:2cm and 3cm);
\draw [-] (0,0) arc (180:0:2cm and 1.3cm);
\node [style=bb,label=below:$\textcolor{red}{2}$] (2) at (0, 0) {};
\node [style=bb,label=below:$\textcolor{red}{1}$] (3) at (2, 0) {};
\node [style=bb,label=below:$\textcolor{red}{1}$] (4) at (4, 0) {};
\node [label=below:$\textcolor{blue}{2}$] (b) at (0, 1.9) {};
\node [label=below:$\textcolor{blue}{0}$] (c) at (2, 1.8) {};
\node [label=below:$\textcolor{blue}{0}$] (c) at (1, 0.4) {};
\node [label=below:$\textcolor{blue}{1}$] (c) at (3, 0.4) {};
\node [label=below:$\textcolor{blue}{1}$] (c) at (2.8, 2.5) {};
\node [label=below:$\textcolor{blue}{0}$] (c) at (2.8, 3.2) {};
\node [label=below:$\textcolor{blue}{1}$] (c) at (2.8, 3.8) {};
\end{tikzpicture}
\\
$b_{\ell}(e)=0$ & $b_{\ell}(e)=1$ & $b_{\ell}(e)=2$
\end{tabular}
}
}
\caption{Example of dynamic flow: (a) signed graph $G$ with positive edge $e$ split into two half-edges, (b) three of the $17$ dynamic integer flows where $b_{\ell}(e)=0,1$, and $2$ so that zero, one and two right positive half-edges are added respectively.}
\label{fig:dynflow}
\end{figure}
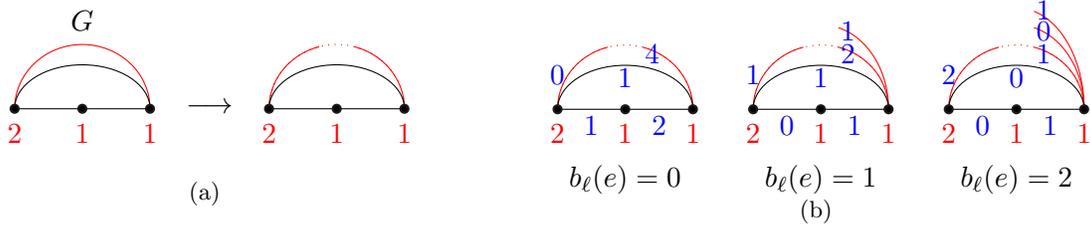

Given a signed graph $G$ on the vertex set $[n+1]$ and ${\bf a}$ a vector in $\mathbb{Z}^{n+1}$, the {\bf dynamic Kostant partition function} $K_G^{\text{dyn}}({\bf a})$ is the number of integer dynamic ${\bf a}$-flows in $G$.  

\begin{proposition} \label{prop:Kdyngs} \cite[Proposition 6.11]{mm} 
The generating series of the dynamic Kostant partition function is
\begin{equation}\label{Kdyngs}
\sum_{{\bf a} \in \mathbb{Z}^{n+1}} K_G^{\text{dyn}}({\bf a}){\bf x}^{\bf a} = {\prod_{(i,j,-) \in E(G)}} (1-x_ix_j^{-1})^{-1} {\prod_{(i,j,+) \in E(G)}} (1-x_i-x_j)^{-1},
\end{equation}
where ${\bf x}^{\bf a}=x_1^{a_1}x_2^{a_2}\cdots x_{n+1}^{a_{n+1}}$.
\end{proposition}

\begin{theorem} \cite[Theorem 6.9]{mm} \label{volD}
Given a loopless connected signed graph $G$ on the vertex set $[n+1]$, let $d_i=\indeg_G(i)-1$ for $i\in \{2,\ldots,n\}$. The normalized volume $\vol(\F_G)$ of the flow polytope associated to the graph $G$ is   
\[
\vol(\F_G(2,0,\ldots,0)) = K_G^{\text{dyn}}(0,d_2, \ldots, d_n,d_{n+1}).
\]  
\end{theorem}  
Theorem \ref{volD} is in general false for graphs with loops, as it already fails for  
\[
G=([3], \{(1,2,-),(1,2,-),(1,2,-),(2,2,+)\}). 
\]
This theorem does not apply to $CRYC_{n+1}=\mathcal{F}_{K_{n+1}^C}(2,0,\ldots,0)$ as $K_{n+1}^C$ has loops.
Nevertheless, we will show in the next Section that the analogue of Theorem \ref{volD} holds for $CRYC_{n+1}$.
Our proof is specific to $CRYC_{n+1}$ and (obviously) cannot be extended to general graphs with loops.

\section{Volumes of  $CRYC_{n+1}$ and $CRYD_{n+1}$ via constant term identities}
\label{sec:volumes}

Suppose that a multi-variable function $f(x_1,\dots,x_n)$ is a Laurent series in $x_i$ considering other variables as constants. Then we denote by $CT_{x_i}f(x_1,\dots,x_n)$ the constant term in the Laurent expansion.   Based on Theorem \ref{volD} it is proved in \cite{mm} that:

\begin{proposition} \cite[Proposition 7.5]{mm} The normalized volume of $CRY D_{n+1}$ is
\begin{equation} \label{eqCRYDMorrislike}
\vol(CRY D_{n+1}) = \CT_{x_{n-1}}\cdots \CT_{x_1} \prod_{i=1}^{n-1}x_i^{-1} (1-x_i)^{-2}  \prod_{{1\leq i<j \leq n-1}} (x_j-x_i)^{-1}(1-x_j-x_i)^{-1} .
\end{equation}
\end{proposition}

Using the above Conjecture \ref{conjcryD} can be rewritten as a constant term identity, and this is how Zeilberger \cite{Z1} proved it; we expand on his proof  in the next section. 
This section is devoted to proving a similar constant term identity for $CRYC_{n+1}$. We know that Theorem \ref{volD} does not apply to this case.
We now show that the analogue of Theorem \ref{volD} holds for $CRYC_{n+1}$. 

\begin{theorem} \label{volC} The normalized volume of $CRY C_{n+1}$ is 

\begin{equation*}  
\vol(CRYC_{n+1}) = CT_{x_{n-1}}CT_{x_{n-2}} \cdots CT_{x_1} \prod_{i=1}^{n-1}x_i^{-1} (1-x_i)^{-2}(1-2x_i)^{-1}  \prod_{{1\leq i<j \leq n-1}} (x_j-x_i)^{-1}(1-x_j-x_i)^{-1}.
\end{equation*}
\end{theorem}

Note, Theorem \ref{volC} differs only in the presence of $\prod_{i=1}^{n-1}(1-2x_i)^{-1}$ from type $D$ as in equation \eqref{eqCRYDMorrislike} 
and Conjecture \ref{conjcryC} states that $\vol(CRYC_{n+1}) =2^{n-1} \vol(CRYD_{n+1}) $.

\medskip

The rest of this section is devoted to the proof of Theorem \ref{volC}. Our proof uses some of the ideas of the proof of Theorem \ref{volD} together with new 
considerations, so we now  review more  background following \cite{mm}.

\subsection{Reduction rules for signed graphs} In this subsection we explain how to  recursively  compute the volume of the 
flow polytope $\F_G(\aa)$ following \cite[Section 4]{mm}. Figure \ref{fig:subdivrules} is  borrowed from \cite{mm} with permission. 
   
   Given   a graph $G$ on the vertex set $[n+1]$ and   $(a, i, -), (i, b, -) \in E(G)$ for some $a<i<b$, let   $G_1$ and $G_2$ be graphs on the vertex set $[n+1]$ with edge sets
  \begin{align*} 
E(G_1)&=E(G)\backslash \{(a, i,-)\} \cup \{(a, b, -)\},  \\
E(G_2)&=E(G)\backslash \{(i, b,-)\} \cup \{(a, b,-)\}. \tag{R1} 
\end{align*}
    Given   a graph $G$ on the vertex set $[n+1]$ and   $(a, i, -), (i, b, +) \in E(G)$ for some $a<i<b$, let   $G_1$ and $G_2$ be graphs on the vertex set $[n+1]$ with edge sets
  \begin{align*} 
E(G_1)&=E(G)\backslash \{(a, i,-)\} \cup \{(a, b, +)\},  \\
E(G_2)&=E(G)\backslash \{(i, b,+)\} \cup \{(a, b, +)\}. \tag{R2} 
\end{align*}

   Given   a graph $G$ on the vertex set $[n+1]$ and   $(a, i, -), (b, i, +) \in E(G)$ for some $a<b<i$, let   $G_1$ and $G_2$ be graphs on the vertex set $[n+1]$ with edge sets
  \begin{align*} 
E(G_1)&=E(G)\backslash \{(a, i, -)\} \cup \{(a, b, +)\},  \\
E(G_2)&=E(G)\backslash \{(b, i, +)\} \cup \{(a, b, +)\}. \tag{R3}
\end{align*}
   Given   a graph $G$ on the vertex set $[n+1]$ and   $(a, i, +), (b, i, -) \in E(G)$ for some $a<b<i$, let   $G_1$ and $G_2$ be graphs on the vertex set $[n+1]$ with edge sets
  \begin{align*} 
E(G_1)&=E(G)\backslash \{(a, i, +)\} \cup \{(a, b, +)\},  \\
E(G_2)&=E(G)\backslash \{(b, i, -)\} \cup \{(a, b, +)\}. \tag{R4}
\end{align*}

 Given   a graph $G$ on the vertex set $[n+1]$ and   $(a, i, -), (a, i, +) \in E(G)$ for some $a<i $, let   $G_1$ and $G_2$ be graphs on the vertex set $[n+1]$ with edge sets
  \begin{align*} 
E(G_1)&=E(G)\backslash \{(a, i, +)\} \cup \{(a, a, +)\},  \\
E(G_2)&=E(G)\backslash \{(a, i, -)\} \cup \{(a, a, +)\}.\tag{R5} 
\end{align*}

Given   a graph $G$ on the vertex set $[n+1]$ and   $(a, i, -), (i, i, +) \in E(G)$ for some $a<i $, let   $G_1$ and $G_2$ be graphs on the vertex set $[n+1]$ with edge sets
  \begin{align*} 
E(G_1)&=E(G)\backslash \{(a, i, -)\} \cup \{(a, i, +)\},  \\
E(G_2)&=E(G)\backslash \{(i, i, +)\} \cup \{(a, i, +)\}\tag{R6} .\end{align*}

    We say that $G$ {\em reduces} to $G_1$ and $G_2$ under the reduction rules (R1)-(R6). We also say in the above cases that we are \textit{reducing at vertex} $i$.  Figure \ref{fig:subdivrules} shows these reduction rules graphically and explains the basic idea that implies that we can use these reductions to dissect flow polytopes. In this paper we will only be using special cases of the results in \cite{mm}, and we state these special cases next.

   \begin{lemma} \label{red} Let $G$ be a signed graph   on the vertex set $[n+1]$ and let  $e_1$ and $e_2$  be two edges of $G$ on which one of the reductions (R1)-(R6) can be performed yielding  graphs  $G_1$ and $G_2$. If  the dimensions of $\F_G(2, 0, \ldots, 0), \F_{G_1}(2, 0, \ldots, 0)$ and $\F_{G_2}(2, 0, \ldots, 0)$ are the same, then 
    $$\vol \F_G(2, 0, \ldots, 0)=\vol \F_{G_1}(2, 0, \ldots, 0) +\vol \F_{G_2}(2, 0, \ldots, 0).$$
If on the other hand only one of $\F_{G_1}(2, 0, \ldots, 0)$ and $\F_{G_2}(2, 0, \ldots, 0)$ is of dimension $\dim \F_G(2, 0, \ldots, 0)$, and the other one has strictly lower dimension, we obtain
 
 $$\vol \F_G(2, 0, \ldots, 0)=\vol \F_{G_i}(2, 0, \ldots, 0),$$ where  $\F_{G_i}(2, 0, \ldots, 0)$ is of dimension $\dim \F_G(2, 0, \ldots, 0)$ (for $i=1$ or $i=2$). 
 
 Moreover, the above are all the possible cases.
\end{lemma}
 
\begin{figure}
\centering
\includegraphics[height=8cm]{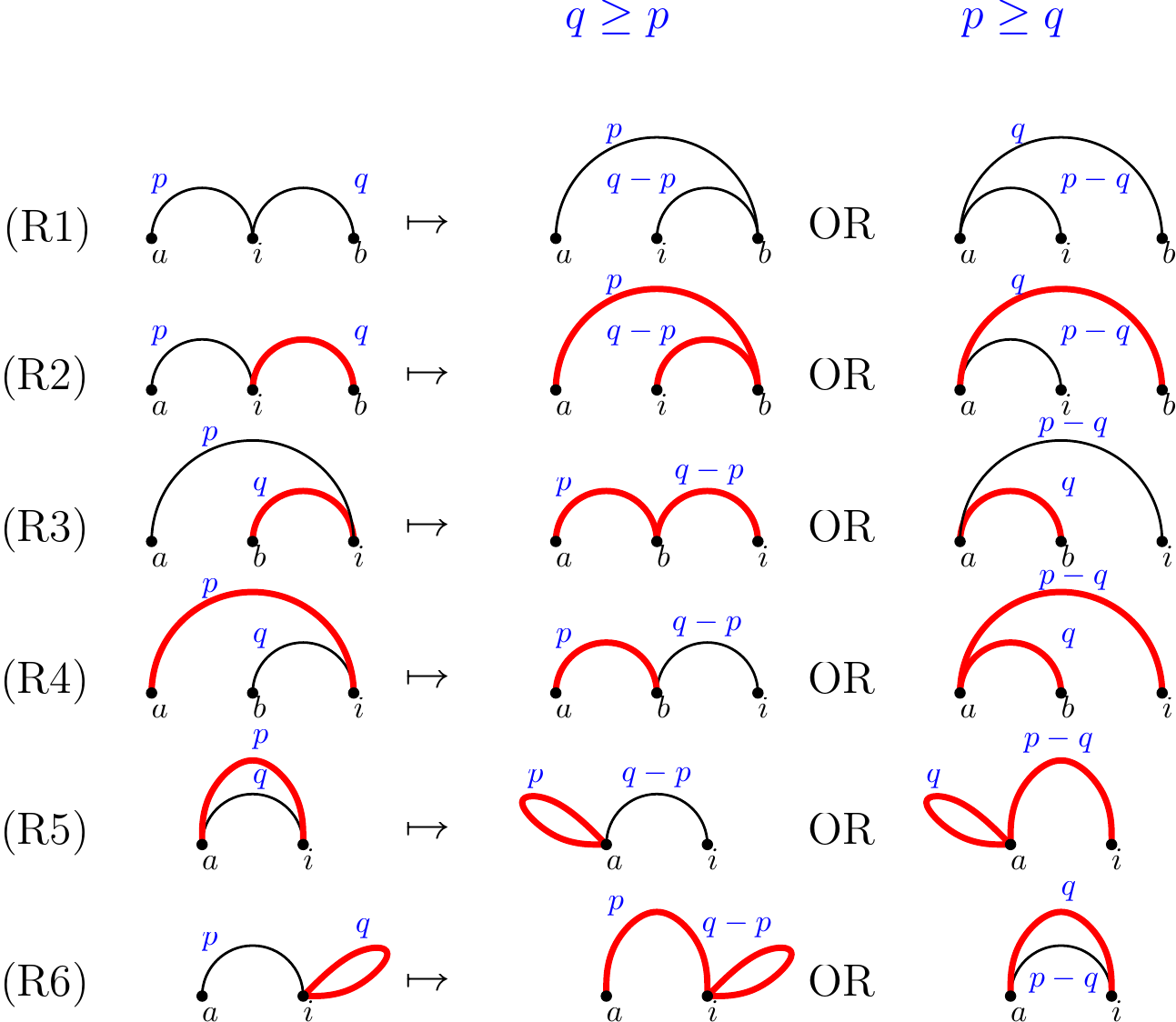}
\caption{Reduction rules from equations (R1)-(R6). The original edges have flow $p$ and $q$. The outcomes have reassigned flows to preserve the original netflow on the vertices.}
 \label{fig:subdivrules}
\end{figure}

Another lemma that follows  from considerations in \cite{mm} is: 

\begin{lemma} \label{loops}  Given a signed graph $G=([n+1], E)$ with loops only at the vertex $1$, let $L\subset E$ be the multiset of its loops. Denote by $G^1=([n+1], E\backslash L)$ the graph obtained from $G$ with its loops at vertex $1$ removed. Then,  
 $$\vol \F_G(2, 0, \ldots, 0)=\vol \F_{G^1}(2, 0, \ldots, 0).$$
\end{lemma} 

\subsection{The proof of Theorem \ref{volC}} We prove a sequence of statements which together imply Theorem \ref{volC}.

\begin{theorem} \label{1} 
We have  
\begin{equation*} \vol(CRYC_{n+1})=\sum_{G \in \mathcal{G}} \vol(\F_G(2,0, \ldots, 0)), \end{equation*} where 
$\mathcal{G}=\{G=([n+1], \cup_{v=2}^{n+1} S_{i_v}^{(v)}) \mid i_v \in [v-1],  v\in [2,n+1]\}$ and $S_k^{(v)}=\{(i,v,+),(i,v,+) \mid i \in [1,k-1]\} \cup \{(k,v,+),(k,v,+), (k,v,-)\}\cup \{(i,v,+),(i,v,-) \mid i \in [k+1,v-1] \}$. 
\end{theorem}

\proof The proof will proceed via the following steps. We will prescribe an order $\mathcal{O}$ of  repeated reductions (R6)   on $K_{n+1}^C$ and its descendants obtained via these reductions until all graphs obtained have loops only at vertex $1$. We will identify the graphs obtained this way whose flow polytopes  are of the same dimension  as $CRYC_{n+1}=\F_{K_{n+1}^C}(2, 0, \ldots, 0)$. Denote this set of graphs by  $\mathcal{G}(1)$. Applying Lemma \ref{red} we can then write $$\vol(CRYC_{n+1})(2,0, \ldots, 0)=\sum_{G \in \mathcal{G}(1)} \vol(\F_G(2,0, \ldots, 0)).$$ 
 We then observe that if we remove the loops at vertex $1$ from the graphs in $\mathcal{G}(1)$, then we exactly obtain the graphs in $\mathcal{G}$. Thereby, by an application of Lemma \ref{loops} we obtain the statement of Theorem \ref{1}.

 Now we prescribe the  order $\mathcal{O}$ of  repeated reductions (R6)   to $G:=K_{n+1}^C$.
 First we reduce at vertex $2$, until there is nothing to reduce at vertex $2$. Next we reduce at  $3$, until there is nothing to reduce at vertex $3$. We continue like this, until we finally reduce at vertex $n+1$, until there is nothing to reduce at vertex $n+1$. Now, we specify the order of reductions at a given vertex $v$. Since we are only using (R6), we are always using the edge $(v,v,+)$ in the reduction and one incoming edge to $v$.    Order the incoming edges by length and start the reductions  from longest edge towards the shortest until we eliminate the loop at vertex $v$.

We now study the reduction order $\mathcal{O}$.  The order $\mathcal{O}$ calls for  applying reduction  (R6) to edges $(1,2, -)$ and $(2,2, +)$ at vertex $2$. Observe that $\dim \F_{G_1}<\dim \F_{K_{n+1}^C}$ and $\dim \F_{G_2}=\dim \F_{K_{n+1}^C}$. One way to see this is to consider reducing $G_1$ and $G_2$ via the reductions (R1)-(R6) until  no more reductions are possible. At this point all graphs have loops only at vertex $1$ and at vertices with no incoming edges.  It is not hard to see that the maximal possible number of loops at vertex $1$ in a  descendant of $G_1$ is strictly less than  the maximal possible number of loops at vertex $1$ in a descendant of $G_2$. Yet, the maximal possible number of loops at $1$ in a descendant of $G_i$ ($i \in [2]$) which cannot be reduced further via  (R1)-(R6) with loops only at $1$ and  at vertices with no incoming edges equals the dimension of $\F_{G_i}$.  Therefore, we proved the following: 

\medskip
\noindent \textbf{Claim (at vertex $2$).}  If the graph $D$  obtained from $K_{n+1}^C$   by  repeatedly performing (R6) as specified by the order $\mathcal{O}$  is such that $\dim \F_D(2,0,\ldots,0)=\dim \F_{K_{n+1}^C}(2,0,\ldots, 0)$ and in $D$ there is no loop at vertex $2$, then $D$ has three edges of the form $(i,2, \e),$ where $i<2$, and these edges must be $(1,2,+), (1,2,+)$ and $(1,2,-)$. 
  \medskip

Next, take $G_2$ (the only $\dim G$-dimensional descendant of $G$ after performing reductions (R6) in order $\mathcal{O}$ at vertex $2$ until all loops at $2$ are eliminated) and  do repeated reductions (R6) at the vertex $3$ according to $\mathcal{O}$. An analogous argument to above gives that:

\medskip
\noindent \textbf{Claim (at vertex $3$).}  If the graph $D$  obtained from $K_{n+1}^C$   by  repeatedly performing (R6) as specified by the order $\mathcal{O}$   is such that $\dim \F_D(2,0,\ldots,0)=\dim \F_{K_{n+1}^C}(2,0,\ldots, 0)$ and in $D$ there is no loop at vertex $3$, then $D$ has   $5$ edges incident to vertex $3$ of the form $(i,3, \e),$ where $i<3$,  and these edges can be:

\begin{itemize}
 
\item[(1)] $(1,3,+), (1,3,+), (1,3,-)$ and $(2,3,+),(2,3,-)$, or  

\item[(2)] $(1,3,+), (1,3,+)$ and $(2,3,+),(2,3,+), (2,3,-).$

\end{itemize}
\medskip

Generalizing straightforwardly, we obtain:

\medskip

\noindent \textbf{Claim (at vertex $v$).}  If the graph $D$  obtained from $K_{n+1}^C$   by  repeatedly performing (R6) as specified by the order $\mathcal{O}$   is such that $\dim \F_D(2,0,\ldots,0)=\dim \F_{K_{n+1}^C}(2,0,\ldots, 0)$ and in $D$ there is no loop at vertex $v$, then $D$  has $2v-1$ edges incident to vertex $v$ of the form $(i,v, \e),$ where $i<v$,  and the set of these edges 
can be:

\begin{itemize}
 
\item[(1)] $S_1^{(v)}=\{(1,v,+), (1,v,+), (1,v,-)\} \cup \{(i,v,+),(i,v,-) \mid i \in [2,v-1] \}$, or  

\item[(2)] $S_2^{(v)}=\{(1,v,+), (1,v,+)\}\cup \{(2,v,+),(2,v,+), (2,v,-)\}\cup \{(i,v,+),(i,v,-) \mid i \in [3,v-1] \}$, or  

$\cdots$

\item[($k$)] $S_k^{(v)}=\{(i,v,+),(i,v,+) \mid i \in [1,k-1]\} \cup \{(k,v,+),(k,v,+), (k,v,-)\}\cup \{(i,v,+),(i,v,-) \mid i \in [k+1,v-1] \}$

$\cdots$

\item[($v-1$)] $S_{v-1}^{(v)}=\{(i,v,+),(i,v,+) \mid i \in [1,v-2]\} \cup \{(v-1,v,+),(v-1,v,+), (v-1,v,-)\}$.
\end{itemize}

\medskip

Thus, by the claims we obtain a description of  $\mathcal{G}(1)$. It is clear that once the loops at vertex $1$ of the graphs in  $\mathcal{G}(1)$ are deleted we obtain the set of graphs in $\mathcal{G}$. This concludes the proof as explained at the beginning of the proof. 
\qed

\begin{theorem} \label{3} 
We have
\begin{equation} \label{G}  K_{K_{n+1}^C}^{dyn}(0,0,1,2,\ldots,n-1)=\sum_{G \in \mathcal{G}} \vol(\F_G(2,0, \ldots, 0)), \end{equation} where 
$\mathcal{G}=\{G=([n+1], \cup_{v=2}^{n+1} S_{i_v}^{(v)}) \mid i_v \in [v-1],  v\in [2,n+1]\}$ and $S_k^{(v)}=\{(i,v,+),(i,v,+) \mid i \in [1,k-1]\} \cup \{(k,v,+),(k,v,+), (k,v,-)\}\cup \{(i,v,+),(i,v,-) \mid i \in [k+1,v-1] \}$.
\end{theorem}

\proof By Theorem \ref{volD}, $\vol(\F_G(2,0, \ldots, 0))=K_G^{dyn}(0, \indeg_G(2)-1, \ldots, \indeg_G(n+1)-1)$ for each $G \in \mathcal{G}$. Thus, equation \eqref{G} is equivalent to:

   \begin{equation} \label{G1}  K_{K_{n+1}^C}^{dyn}(0,0,1,2,\ldots,n-1)=\sum_{G \in \mathcal{G}} K_G^{dyn}(0, \indeg_G(2)-1, \ldots, \indeg_G(n+1)-1). \end{equation} 

We prove \eqref{G1} by exhibiting a bijection between the dynamic Kostant partition functions counted  on the left hand side and those counted on the right hand side.

First note that 
 $$\{(0, \indeg_G(2)-1, \ldots, \indeg_G(n+1)-1)\mid   G \in \mathcal{G}\}=\{(0,a_2,a_3, \ldots, a_{n+1}) \mid 0\leq a_v\leq v-2, v \in [2, n+1]\},$$ where the equality holds both as sets and multisets. Thus, once we are given a vector  $\aa=(0,a_2,a_3, \ldots, a_{n+1}) $ with $0\leq a_v\leq v-2$, $v \in [2, n+1]$, it uniquely determines the graph $G \in \mathcal{G}$ such that $(0, \indeg_G(2)-1, \ldots, \indeg_G(n+1)-1)=(0,a_2,a_3, \ldots, a_{n+1})$. In the following we denote this unique graph from $\mathcal{G}$ by $G_{(0,a_2,a_3, \ldots, a_{n+1})}=([n+1], \cup_{v=2}^{n+1} S_{v-a_v-1}^{(v)})$.

Given a dynamic integer flow $f_{\aa}$  on $G=G_{(0,a_2,a_3, \ldots, a_{n+1})}$ with netflow vector $(0,a_2,a_3, \ldots, a_{n+1})$ for some  $0\leq a_v\leq v-2$, $v \in [2, n+1]$, we now specify how to  construct  a dynamic integer flow  $g(f_{\aa})$ on $K_{n+1}^C$ with 
netflow vector $(0,0,1,2,\ldots,n-1)$. Our description involves several steps.

\medskip

\noindent \textit{Notational convention.} A positive edge  $(i,j,+)$ is considered a  left and a right half edge. We denote by $(i,j,+)_l$ and $(i,j,+)_r$ the left and right half of edge $(i,j,+)$. In case there are more left or right half edges we add superscripts; e.g.   $(i,j,+)_r$  and $(i,j,+)^1_r$ are two different right half edges. Given a dynamic flow on a graph $G$ we might have added positive right half edges with flows to $G$; we denote the graph with these positive right half edges added by $G^{dyn}$.
\medskip

Let us fix $v\in[2,n+1]$. We define the flows of $g(f_\aa)$ on the positive half edges 
$(i,v,+)_r$ and $(i,v,+)_l$ for $1\le i\le v$ and the negative edges $(i,v,-)$ for $1\le i\le v-1$ as follows. 

Let $k(v)=v-a_v-1$. Then $1\le k(v)\le v-1$ and for $1\le i\le v-1$, the edges between $i$ and $v$ in $G$ are precisely
\begin{itemize}
\item two positive edges $(i,v,+)$ and $(i,v,+)^1$ if $1\le i\le k(v)-1$,
\item two positive edges $(i,v,+)$ and $(i,v,+)^1$ and one negative edge $(i,v,-)$ if $i=k(v)$,
\item one positive edge $(i,v,+)$ and one negative edge $(i,v,-)$ if $k(v)+1\le i\le v-1$.
\end{itemize}

First, observe that $G$ and $K_{n+1}^C$ have common edges: $(i,v,+)$ for $1\le i\le v-1$ and $(i,v,-)$ for $k(v)\le i\le v-1$. We define the flows of $g(f_\aa)$ related to these edges to be the same as those of $f_\aa$. In other words,
\begin{align*}
g(f_\aa)(i,v,+)_l &= f_\aa(i,v,+)_l \qquad \mbox{for $1\le i\le v-1$},  \\
 g(f_\aa)(i,v,+)_r &= f_\aa(i,v,+)_r \qquad \mbox{for $1\le i\le v-1$},  \\
 g(f_\aa)(i,v,-) &= f_\aa(i,v,-) \qquad \mbox{for $k(v)\le i\le v-1$}.
\end{align*}
For the new right half edges at $v$ we just transfer whatever the value of $f_{\aa}$ is on these new right half edges to $g(f_{\aa})$ on the corresponding new right half edges of $(K_{n+1}^{C})^{dyn}$.

Now we need to consider the half edges $(i,v,+)^1_l$ and $(i,v,+)^1_r$ for $1\le i\le k(v)$ in $G$. 

Firstly, we set \[ g(f_\aa)(v,v,+)_l = k(v)-1, \]
 thereby creating new half edges 
$(v,v,+)_r^1, \ldots, (v,v,+)_r^{k(v)-1}$.
Then we define
 \begin{align*}
g(f_\aa)(v,v,+)_r &= f_\aa(1,v,+)_r^1,  \\
 g(f_\aa)(v,v,+)_r^i &= f_\aa(i+1,v,+)_r^1 \qquad \mbox{for $1\le i\le k(v)-1$}.
 \end{align*}

Secondly, we define
\begin{equation}
  \label{eq:g(i,v,-)}
g(f_\aa)(i,v,-) = f_\aa(i,v,+)_l^1 \qquad \mbox{for $1\le i\le k(v)-1$},   
\end{equation}
and increase the value of $g(f_\aa)(k(v),v,-)$, which has been defined above, by $f_\aa(k(v),v,+)_l^1$, so that
\begin{equation}
  \label{eq:g(k(v),v,-)}
g(f_\aa)(k(v),v,-) = f_\aa(k(v),v,-) + f_\aa(k(v),v,+)_l^1.
\end{equation}

Finally, we increase the value of $g(f_\aa)(v,v,+)_l$ so that
\begin{equation}
  \label{eq:g(v,v,+)}
g(f_\aa)(v,v,+)_l = k(v)-1 + \sum_{i=1}^{k(v)} f_\aa(i,v,+)_l^1.  
\end{equation}
This creates $\sum_{i=1}^{k(v)} f_\aa(i,v,+)_l^1$ new right half edges at $v$ in $(K_{n+1}^{C})^{dyn}$. 
We  transfer the values of $f_\aa$ on the same number of new right half edges in $G^{dyn}$ created by 
the values $f_\aa(i,v,+)_l^1$ for $1\le i\le k(v)$ to the new right half edges just created in $(K_{n+1}^{C})^{dyn}$.

Note that the netflow of $g(f_\aa)$ at $v$ is $k(v)-1+a_v=v-2$. Thus $g(f_\aa)$ is a dynamic integer flow on $K_{n+1}^C$ with netflow vector $(0,0,1,2,\ldots,n-1)$.

It is not hard to check that the map $f_\aa\mapsto g(f_\aa)$ is invertible. We now explain how to recover the vector $\aa=(0,a_2,a_3,\dots,a_{n+1})$ from $g(f_\aa)$. For each $v\in [2,n+1]$, we  find $a_v$ as follows. As before,  $k(v)=v-1-a_v$. By \eqref{eq:g(i,v,-)}, \eqref{eq:g(k(v),v,-)}, and \eqref{eq:g(v,v,+)}, we have
\[
g(f_\aa)(v,v,+)_l = k(v)-1 -f_\aa(k(v),v,-1) + \sum_{i=1}^{k(v)} g(f_\aa)(i,v,-).  
\]
Thus
\begin{equation}
  \label{eq:kv1}
g(f_\aa)(v,v,+)_l \le k(v)-1 + \sum_{i=1}^{k(v)} g(f_\aa)(i,v,-).    
\end{equation}
On the other hand, by \eqref{eq:g(i,v,-)} and \eqref{eq:g(v,v,+)}, we have
\begin{equation}
  \label{eq:kv2}
g(f_\aa)(v,v,+)_l \ge k(v)-1 + \sum_{i=1}^{k(v)-1} g(f_\aa)(i,v,-).    
\end{equation}
By \eqref{eq:kv1} and \eqref{eq:kv2}, we obtain that $k(v)$ is the unique integer $t$ satisfying
\[
t-2+ \sum_{i=1}^{t-1} g(f_\aa)(i,v,-) <
g(f_\aa)(v,v,+)_l \le t-1 + \sum_{i=1}^{t} g(f_\aa)(i,v,-).    
\]

Therefore, we can recover $a_v=v-1-k(v)$ from $g(f_\aa)$. Once $\aa$ is obtained, it is easy to recover $f_\aa$ from $g(f_\aa)$. This is a desired bijection and the proof is completed. 

\qed

\begin{example}
We give here a simple example of this construction. If ${\mathbf a}=(0,\ldots ,0)$, the graph $G$ has edges $(i,i+1,-)$, $(i,j,+)$, $(i,j,+)^1$
for $1\le i<j \le n+1$. The unique dynamic flow is such that the edges $(i,j,+)$, $(i,j,+)^1$ become half edges
$(i,j,+)_l$, $(i,j,+)_l^1$ and  $(i,j,+)_r$, $(i,j,+)_r^1$. Every (half) edge $e$ has  flow $f_{\mathbf a}(e)=0$.
The corresponding dynamic flow on  $K_{n+1}^C$ is such that each (half) edge $e=(i,j,+)_l$, $(i,j,+)_r$, or $(i,j,-)$
has flow $g(f_{\mathbf a})(e)=0$ for $1\le i<j \le n+1$. The half edges $(v,v,+)_l$ for $2\le v\le n+1$ are such that
$g(f_{\mathbf a})(v,v,+)_l=v-2$ and the half edges $(v,v,+)_l^i$ have flow $g(f_{\mathbf a})(v,v,+)_r^i=0$ for $1\le i\le v-1$.
\end{example}

We finally write $K_{K_{n+1}^C}^{dyn}(0,0,1,2,\ldots,n-1)$ as a constant term identity~:
\begin{lemma} \label{2}
We have
\begin{multline*}
K_{K_{n+1}^C}^{dyn}(0,0,1,2,\ldots,n-1)\\
=CT_{x_{n-1}}CT_{x_{n-2}} \cdots CT_{x_1} \prod_{i=1}^{n-1}x_i^{-1} (1-x_i)^{-2}(1-2x_i)^{-1}  \prod_{{1\leq i<j \leq n-1}} (x_j-x_i)^{-1}(1-x_j-x_i)^{-1}  .
\end{multline*}
\end{lemma}

\begin{proof}

By Proposition~\ref{prop:Kdyngs}  we get that 
\begin{multline*}
K_{K_{n+1}^C}^{dyn}(0,0,1,2,\ldots,n-1) \\=  [x_3^1x_4^2\cdots x_{n+1}^{n-1}] \CT_{x_2}\CT_{x_1} {\prod_{1\leq i<j\leq n+1}} (1-x_ix_j^{-1})^{-1}(1-x_i-x_j)^{-1} {\prod_{1\leq i\leq n+1}} (1-2x_i)^{-1}.
\end{multline*}
Then by plugging in $x_1=x_2=0$ and relabeling the variables
$x_m\mapsto x_{m-2}$  gives:
\begin{multline*}
K_{K_{n+1}^C}^{dyn}(0,0,1,2,\ldots,n-1) \\= [x_1^1x_2^2\cdots x_{n-1}^{n-1}] {\prod_{1\leq i<j\leq n-1}} (1-x_ix_j^{-1})^{-1}(1-x_i-x_j)^{-1}\prod_{{1\leq i\leq n-1}} (1-x_i)^{-2} {\prod_{1\leq i\leq n-1}} (1-2x_i)^{-1}.
\end{multline*}
The above equation is equivalent to the desired expression:
\begin{multline*}
K_{K_{n+1}^C}^{dyn}(0,0,1,2,\ldots,n-1) \\= \CT_{x_{n-1}}\CT_{x_{n-2}} \cdots \CT_{x_1} \prod_{i=1}^{n-1}x_i^{-1} (1-x_i)^{-2}  (1-2x_i)^{-1}\prod_{{1\leq i<j \leq n-1}} (x_j-x_i)^{-1}(1-x_j-x_i)^{-1}.
\end{multline*}
\end{proof}
\medskip

\noindent {\it Proof of Theorem \ref{volC}.} Immediate corollary of Theorems \ref{1} and \ref{3} and Lemma \ref{2}.

\section{Proofs of Conjectures \ref{conjcryD} and \ref{conjcryC}}
\label{sec:vol}

In this section we prove Conjecture
\ref{conjcryC} and give a detailed proof of  Zeilberger's theorem, formerly Conjecture \ref{conjcryD}. 
We begin by recalling Zeilberger's approach via the Morris' identity for proving  the volume formula of $CRY_n$ ~\cite{Z}. 

\begin{lemma}[Morris'
identity,  \cite{Z}] For nonnegative integers $a,b$ and a positive half
integer $c$, we have
\begin{equation}
  \label{eq:zeilberger}
\CT_{x_n}\CT_{x_{n-1}}\cdots \CT_{x_1} \prod_{i=1}^n (1-x_i)^{-a} x_i^{-b} 
\prod_{1\le i<j\le  n}(x_j-x_i)^{-2c} 
= \frac{1}{n!} \prod_{j=0}^{n-1} \frac{\Gamma(a+b+(n-1+j)c)\Gamma(c)}
{\Gamma(a+jc)\Gamma(c+jc)\Gamma(b+jc+1)}.
\end{equation}
\end{lemma}

In  \cite{Z} Zeilberger proved the volume formula for  $CRY_n$ by
showing that, when we set $a=2,b=0,c=1/2$ in \eqref{eq:zeilberger}, we have
\begin{equation}
  \label{eq:cat}
\CT_{x_n}\CT_{x_{n-1}}\cdots \CT_{x_1} \prod_{i=1}^n (1-x_i)^{-2}
\prod_{1\le i<j\le  n}(x_j-x_i)^{-1} =
\frac{1}{n!} \prod_{j=0}^{n-1}
\frac{\Gamma(\frac{n+3+j}2)\Gamma(\frac 12)}
{\Gamma(\frac{4+j}2)\Gamma(\frac{1+j}2)\Gamma(\frac{2+j}2)}=  
\prod_{k=1}^n \Cat(k).
\end{equation}

Using \eqref{eq:zeilberger}, we will prove the following theorem.

\begin{theorem}\label{thm:C} 
For a nonnegative integer $a$ and a positive half integer
  $c$, we have
\begin{multline*}
\CT_{x_n}\CT_{x_{n-1}}\cdots \CT_{x_1}  \prod_{j=1}^n x_j^{-a+1}(1-x_j)^{-a}
(1-2x_j)^{-b} \prod_{1\le j<k\le  n} (x_j-x_k)^{-2c} (1-x_j-x_k)^{-2c}  
\\= 2^{2an+4c\binom n2-2n} 
\frac{1}{n!} \prod_{j=0}^{n-1} \frac{\Gamma(a+\frac{b-1}2+(n-1+j)c)\Gamma(c)}
{\Gamma(\frac{b+1}2+jc)\Gamma(c+jc)\Gamma(a+jc)}.
\end{multline*}
\end{theorem}

Before proving Theorem~\ref{thm:C} we show how this theorem implies
Conjectures~\ref{conjcryD} and \ref{conjcryC}. 

\begin{proof}[Proof of Conjecture~\ref{conjcryD}]
If $a=2,b=0,c=1/2$ in Theorem~\ref{thm:C}, we have
\begin{multline}\label{eq:1}
\CT_{x_n}\CT_{x_{n-1}}\cdots \CT_{x_1}  \prod_{j=1}^n x_j^{-1}(1-x_j)^{-2}
\prod_{1\le j<k\le  n} (x_j-x_k)^{-1} (1-x_j-x_k)^{-1}  
\\= \frac{2^{n^2+n}}{n!}\prod_{j=0}^{n-1}
\frac{\Gamma(\frac{n+2+j}2)\Gamma(\frac 12)}
{\Gamma(\frac{1+j}2)\Gamma(\frac{1+j}2)\Gamma(\frac{4+j}2)}.
\end{multline}
Since
\[
\prod_{j=0}^{n-1}
\frac{\Gamma(\frac{n+3+j}2)\Gamma(\frac 12)}
{\Gamma(\frac{4+j}2)\Gamma(\frac{1+j}2)\Gamma(\frac{2+j}2)}
\left/
\prod_{j=0}^{n-1}
\frac{\Gamma(\frac{n+2+j}2)\Gamma(\frac 12)}
{\Gamma(\frac{1+j}2)\Gamma(\frac{1+j}2)\Gamma(\frac{4+j}2)}=  
\right.
\frac{\Gamma(\frac{2n+2}2)\Gamma(\frac{1}2)}
{\Gamma(\frac{n+2}2)\Gamma(\frac{n+1}2)}
=2^n, 
\]
the right hand side of \eqref{eq:1} is equal to
\[
\frac{2^{n^2}}{n!} \prod_{j=0}^{n-1}
\frac{\Gamma(\frac{n+3+j}2)\Gamma(\frac 12)}
{\Gamma(\frac{4+j}2)\Gamma(\frac{1+j}2)\Gamma(\frac{2+j}2)} =
2^{n^2}\prod_{k=1}^n \Cat(k),
\]
where \eqref{eq:cat} is used for the last equation.  Thus
Conjecture~\ref{conjcryD} follows from \eqref{eqCRYDMorrislike}
\end{proof}

\begin{proof}[Proof of Conjecture~\ref{conjcryC}]
If $a=2,b=1,c=1/2$ in Theorem~\ref{thm:C}, we have
\begin{multline*}
\CT_{x_n}\CT_{x_{n-1}}\cdots \CT_{x_1}  \prod_{j=1}^n x_j^{-1}(1-x_j)^{-2}
(1-2x_j)^{-1} \prod_{1\le j<k\le  n} (x_j-x_k)^{-1} (1-x_j-x_k)^{-1}  
\\= \frac{2^{n^2+n}}{n!}\prod_{j=0}^{n-1}
\frac{\Gamma(\frac{n+3+j}2)\Gamma(\frac 12)}
{\Gamma(\frac{2+j}2)\Gamma(\frac{1+j}2)\Gamma(\frac{4+j}2)}
=2^{n^2+n}\prod_{k=1}^n \Cat(k),
\end{multline*}
where \eqref{eq:cat} is used for the last equation.  Thus
Conjecture~\ref{conjcryC} follows from Theorem~\ref{volC}
\end{proof}

For the rest of this section we prove Theorem~\ref{thm:C}. The idea is
to change constant terms into contour integrals and consider several
changes of variables.

For a function $f(z)$ with a Laurent series expansion at $z$, we
denote by $\CT_z f(z)$ the constant term of the Laurent expansion of
$f(z)$ at $0$. In other words, if
$f(z)=\sum_{n=-\infty}^\infty a_n z^n$, then $\CT_z f(z) = a_0$.  By
Cauchy's integral formula, if $f(z)$ has a Laurent series expansion at
$0$, we have
\begin{equation}
  \label{eq:CT}
\CT_z f(z) = \frac{1}{2\pi i} \oint_C \frac{f(z)} z dz,
\end{equation}
where $C$ is the circle $\{z: |z|=\epsilon\}$ oriented
counterclockwise for a real number $\epsilon >0$ such that $f(z)$ is
holomorphic inside $C$ except $0$.  Thus \eqref{eq:zeilberger} can be rewritten
as
\begin{multline}
  \label{eq:zeilberger2}
\frac{1}{(2\pi i)^n}\oint_{C_n}\dots\oint_{C_1}  \prod_{j=1}^n (1-x_j)^{-a} x_j^{-b-1} 
\prod_{1\le j<k\le  n}(x_k-x_j)^{-2c}  dx_1\cdots dx_n\\
= \frac{1}{n!} \prod_{j=0}^{n-1} \frac{\Gamma(a+b+(n-1+j)c)\Gamma(c)}
{\Gamma(a+jc)\Gamma(c+jc)\Gamma(b+jc+1)},
\end{multline}
where $C_j$ is the circle $\{z: |z|=j\epsilon\}$ oriented
counterclockwise for a real number $0<\epsilon<\frac 1n$. In
\eqref{eq:zeilberger2} $a$ can be any positive real number. 

\begin{proof}[Proof of Theorem~\ref{thm:C}]
  Let $L$ denote the left hand side of the identity in the theorem, i.e.,
\[
L = \CT_{x_n}\CT_{x_{n-1}}\cdots \CT_{x_1}  \prod_{j=1}^n
  x_j^{-a+1}(1-x_j)^{-a} (1-2x_j)^{-b}
\prod_{1\le j<k\le  n} (x_k-x_j)^{-2c} (1-x_k-x_j)^{-2c}.
\]
By \eqref{eq:CT}, $L$ is equal to
\[
\frac{1}{(2\pi i)^n}\oint_{C_n}\cdots\oint_{C_1}
\prod_{j=1}^n x_j^{-a}(1-x_j)^{-a} (1-2x_j)^{-b}
\prod_{1\le j<k\le  n} (x_k-x_j)^{-2c} (1-x_k-x_j)^{-2c}
dx_1\cdots dx_n,
\]
where $C_j$ is the circle $\{z: |z|=j\epsilon\}$ oriented
counterclockwise for a real number $0<\epsilon<\frac{1}{2n}$.  We will
express $L$ as a constant multiple of the contour integral in
\eqref{eq:zeilberger2} by using changes of variable 3 times.

Using the change of variables $x_j = \frac{1-z_j}2$ or $z_j=1-2x_j$ to
the above integral, we have 
\begin{multline*}
L = \frac{1}{(2\pi i)^n}\oint_{C'_n}\cdots\oint_{C'_1}
\prod_{j=1}^n \left(\frac{1-z_j}2\right)^{-a} 
\left(\frac{1+z_j}2\right)^{-a}  z_j^{-b}\\
\times \prod_{1\le j<k\le  n} \left(\frac{-z_k+z_j}2\right)^{-2c}
\left(\frac{z_k+z_j}2\right)^{-2c} (-2)^{-n} dz_1\cdots dz_n\\
=\frac{(-1)^n2^{2an+4c\binom n2-n}}{(2\pi i)^n}\oint_{C'_n}\cdots\oint_{C'_1}
\prod_{j=1}^n (1-z_j^2)^{-a} z_j^{-b}
\prod_{1\le j<k\le  n} (z_j^2-z_k^2)^{-2c}
dz_1\cdots dz_n,
\end{multline*}
where $C'_j$ is the circle $\{z: |z-1|=2j\epsilon\}$ oriented
counterclockwise.

Using the change of variables $z_j^2 = y_j$ or $z_j=y_j^{1/2}$, we have
\[
L=\frac{(-1)^n2^{2an+4c\binom n2-2n}}{(2\pi i)^n}\oint_{C''_n}\cdots\oint_{C''_1}
\prod_{j=1}^n (1-y_j)^{-a} y_j^{-b/2} y_j^{-1/2}
\prod_{1\le j<k\le  n} (y_j-y_k)^{-2c}
dy_1\cdots dy_n,
\]
where $C''_j$ is the circle $\{z: |z-1|=4j\epsilon\}$ oriented
counterclockwise. This is because if $C_j'$ is parametrized by
$1+2j\epsilon e^{i\theta}$ for $0\le \theta\le 2\pi$, then the image
of $C_j'$ under the map $z\mapsto z^2$ can be parametrized by
$1+4j\epsilon e^{i\theta} + 4j^2\epsilon^2 e^{2i\theta}$ for $0\le
\theta\le 2\pi$. Since we can make $\epsilon$ arbitrarily close to $0$, we can deform this
image to the circle $C_j''$ without changing the value of the contour integral.

Using the change of variables $t_j = 1-y_j$, we have
\begin{multline*}
L = \frac{2^{2an+4c\binom n2-2n}}{(2\pi i)^n} \oint_{C'''_n}\cdots\oint_{C'''_1}
\prod_{j=1}^n t_j^{-a} (1-t_j)^{-(b+1)/2}
\prod_{1\le j<k\le  n} (t_k-t_j)^{-2c}
dt_1\cdots dt_n \\
=2^{2an+4c\binom n2-2n} \CT_{x_n} \cdots \CT_{x_1}
\prod_{j=1}^n x_j^{-a+1} (1-x_j)^{-(b+1)/2}
\prod_{1\le j<k\le  n} (x_k-x_j)^{-2c},
\end{multline*}
where $C'''_j$ is the circle $\{z: |z|=4j\epsilon\}$ oriented
counterclockwise.  Using \eqref{eq:zeilberger2} we finish the proof.
\end{proof}

\section{Conclusion}
\label{sec:conc}

The link between the Kostant partition function of graphs
and the volume of their flow polytopes has been established a decade ago
\cite{BV}. A generalization of this  correspondence via dynamic Kostant partition functions was demonstrated for loopless signed graphs in \cite{mm}.
In this paper we showed among others that dynamic Kostant partition functions can be used for  certain signed graph with loops to obtain the  volume of their associated flow polytope with netflow vector $(2,0, \ldots, 0)$ analogously to the loopless case. This is not true for all signed graphs
with loops. We leave as an open problem the classification of signed graphs
with loops where the volume of the associated flow polytope with netflow vector $(2,0, \ldots, 0)$
is equal to the corresponding dynamic  Kostant partition function evaluation. More broadly,   is there an appealing  further generalization of the Kostant partition function that would work for calculating the volume of the flow polytope of any signed graph (and netflow vector)?  Finally, it would be very interesting to gain a unified insight into which flow polytopes have nice
product formulas for their volume and why. See this paper and
\cite{CKM,M, mm, tesler} for examples of such nice formulas.

\section*{Acknowledgements} This work started during a stay of the second and third authors at the Universit\'e Paris 7  Diderot. The third author 
is grateful for the invitation from, support of and hospitality of the Universit\'e Paris 7. 
Corteel is partially supported by the project Emergences ``Combinatoire \`a Paris''. Kim is partially supported by the National Research Foundation of Korea 
(NRF) grants (NRF-2016R1D1A1A09917506) and (NRF-2016R1A5A1008055). M\'esz\'aros is partially supported by a National Science Foundation Grant  (DMS 1501059).


\begin{thebibliography}{1}

\bibitem{BV}
W. Baldoni, and M. Vergne, Kostant partitions functions and flow polytopes. 
Transform. Groups 13 (2008), no. 3-4, 447–-469. 

\bibitem{cry}
C.S. Chan, D.P. Robbins, and D.S. Yuen.
\newblock On the volume of a certain polytope.
\newblock {\em Experiment. Math.}, 9(1):91--99, 2000.

\bibitem{CKM}
S. Corteel, J.S. Kim, K. M\'esz\'aros, Flow polytopes with Catalan volumes,
C. R. Acad. Sci. Paris Sér. I Math., 355 (2017), no. 3, 248--259. 

\bibitem{M}
K. M\'esz\'aros, Product formulas for volumes of flow polytopes. Proc. Amer. Math. Soc. 143 (2015), no. 3, 937--954.

\bibitem{mm}
K.~M\'esz\'aros and A.~H. Morales.
\newblock Flow polytopes of signed graphs and the {K}ostant partition function.
\newblock {\em International Mathematical Research Notices}, (2015), no. 3, 830--871.


\bibitem{tesler}
K.~M\'esz\'aros, A. H. Morales, B. Rhoades,  
\newblock{The polytope of {T}esler matrices}. 
\newblock {\em Selecta Mathematica}, 23 (2017), no. 1, 425--454.

\bibitem{WM}
W.G. Morris.
\newblock {\em Constant Term Identities for Finite and Affine Root Systems:
  Conjectures and Theorems}.
\newblock PhD thesis, University of Wisconsin-Madison, 1982.

\bibitem{Z}
D.~Zeilberger.
\newblock Proof of a conjecture of {C}han, {R}obbins, and {Y}uen.
\newblock {\em Electron. Trans. Numer. Anal.}, 9:147--148, 1999.

\bibitem{Z1}
D.~Zeilberger.
\newblock Sketch of a {P}roof of an {I}ntriguing {C}onjecture of {K}arola
  {M}eszaros and {A}lejandro {M}orales {R}egarding the {V}olume of the $D_n$
  {A}nalog of the {C}han-{R}obbins-{Y}uen {P}olytope ({O}r: {T}he
  {M}orris-{S}elberg {C}onstant {T}erm {I}dentity {S}trikes {A}gain!).
\newblock {\em http://arxiv.org/abs/1407.2829}, 2014.

\end{thebibliography}
\end{document}